\documentclass[10pt,reqno]{amsart}
\usepackage{amsfonts}
\usepackage{amsmath,amsthm,amssymb,amsfonts,amscd}
\usepackage{mathrsfs}
\usepackage{lipsum}
\usepackage{bbding}
\usepackage{graphicx,latexsym}
\usepackage{hyperref}

\newcommand{\bea}{\begin{eqnarray}}
\newcommand{\eea}{\end{eqnarray}}
\newcommand{\bna}{\begin{eqnarray*}}
\newcommand{\ena}{\end{eqnarray*}}

\numberwithin{equation}{section}

\theoremstyle{plain}
\newtheorem{theorem}{Theorem}
\newtheorem{lemma}{Lemma}

\newtheorem{proposition}{Proposition}

\theoremstyle{definition}

\newtheorem{remark}{Remark}

\renewcommand{\Re}{\operatorname{Re}}

\begin{document}

\title{Bounds for $\rm GL_2\times GL_2$ $L$-functions in depth aspect}

\author{Qingfeng Sun}

\begin{abstract}
Let $f$ and $g$ be holomorphic or Maass cusp forms
for $\rm SL_2(\mathbb{Z})$ and let $\chi$ be
a primitive Dirichlet character of prime power conductor $\mathfrak{q}=p^{\kappa}$ with $p$ prime
and $\kappa>12$. A subconvex bound for the central values of the Rankin-Selberg $L$-functions
$L(s,f\otimes g \otimes \chi)$ is proved in the depth-aspect
$$
L\left(\frac{1}{2},f\otimes g \otimes \chi\right)\ll_{f,g,\varepsilon}
p^{3/4}\mathfrak{q}^{15/16+\varepsilon}.
$$
\end{abstract}



\maketitle

%\tableofcontents

\section{Introduction}

Let $f$ and $g$ be holomorphic or Maass cusp forms
for $\rm SL_2(\mathbb{Z})$ with normalized Fourier coefficients
$\lambda_f(n)$ and $\lambda_g(n)$ (such that $\lambda_f(1)=1$ and $\lambda_g(1)=1$), respectively.
Let $\chi$ be a primitive Dirichlet character of prime power conductor
$\mathfrak{q}=p^{\kappa}$ with $p$ prime.
The $L$-function associated with $f$ and $g\otimes\chi$ is given by
\bna
L(s,f\otimes g \otimes \chi)=L(2s,\chi^2)\sum_{n=1}^{\infty}
\frac{\lambda_f(n)\lambda_g(n)\chi(n)}{n^s}
\ena
for $\Re(s)>1$, which can be analytically extended to $\mathbb{C}$ and satisfies
a functional equation relating $s$ and $1-s$. The approximate functional equation
and the Phragmen-Lindel\"{o}f principle imply that
$L(1/2,f\otimes g \otimes \chi)\ll_{f,g,\varepsilon} \mathfrak{q}^{1+\varepsilon}$
for any $\varepsilon>0$, which is the convexity bound in the depth aspect.
The purpose of this paper is to prove the following subconvexity bound.

\begin{theorem}
Let $f$ and $g$ be holomorphic or Maass cusp forms
for $\rm SL_2(\mathbb{Z})$ and let $\chi$
a primitive Dirichlet character of prime power conductor
$\mathfrak{q}=p^{\kappa}$ with $\kappa\geq 12$. We have
$$
L\left(\frac{1}{2},f\otimes g \otimes \chi\right)\ll_{f,g,\varepsilon}
p^{3/4}
\mathfrak{q}^{15/16+\varepsilon}
$$
for any $\varepsilon>0$.
\end{theorem}

In the $\rm GL_1$ case, by developing a general new theory of estimation of
short exponential sums involving $p$-adic analytic phases,
Mili\'{c}evi\'{c} \cite{Mili} proved a sub-Weyl subconvexity bound for the central values $L(1/2,\chi)$
of Dirichlet $L$-functions of the form
$$
L(1/2,\chi)\ll p^r\mathfrak{q}^{\theta}(\log q)^{1/2}
$$
with a fixed $r$ and $\theta>\theta_0\approx 0.1645$.
By introducing further $p$-adic tools, Blomer and Mili\'{c}evi\'{c} \cite{BM}
also considered the $\rm GL_2$ case and showed that
\bea
L\left(1/2+it,f\otimes\chi\right)\ll_{p,t,f,\varepsilon}\mathfrak{q}^{1/3+\varepsilon},
\eea
where the implied constant on $p$ and $t$ is explicit and polynomial. It worth noting that
Munshi and Singh \cite{Munshi-Singh} proved the same result using the approach in \cite{Munshi 1}.
Letang \cite{Le} also obtained depth-aspect subconvexity for $\rm GL_2$ by investigating the second
integral moments of families of automorphic $L$-functions.
Using the ideas in \cite{Munshi 1}, the author \cite{SZ} extended Munshi and Singh's results
in \cite{Munshi-Singh} to the $\rm GL_3$ case and proved that for
$\pi$ a Hecke-Maass cusp form for $\rm SL_3(\mathbb{Z})$,
$$
L\left(\frac{1}{2},\pi\otimes \chi\right)\ll_{\pi,\varepsilon}p^{3/4}
\mathfrak{q}^{3/4-3/40+\varepsilon}
$$
for any $\varepsilon>0$. Theorem 1 gives a depth-aspect subconvex bound for
the $\rm GL_2\times GL_2$ case.

\medskip

We also note that the bound in Theorem 1 can be compared with the recent $t$-aspect subconvexity
for $\rm GL_2\times GL_2$ L-functions in
Acharya, Sharma and Singh \cite{ASS} that
$$
L(1/2+it, f\otimes g)\ll_{f,\varepsilon}(1+|t|)^{15/16+\varepsilon}.
$$

To prove Theorem 1, we will use similar ideas as in \cite{Munshi 1},
\cite{Munshi-Singh} and \cite{SZ}. Moreover, since the holomorphic case is usually easier,
to simplify the argument, we prove Theorem 1 only for the case of Maass forms.
\begin{remark}
The ideas of treating character sums in this paper can be applied to prove a $\rm GL_3\times GL_2$ subconvexity bound
of the form
\bna
L\left(\frac{1}{2},\pi\otimes f\otimes \chi\right)\ll_{p,\pi,f,\varepsilon}
\mathfrak{q}^{3/2-3/20+\varepsilon},
\ena
where the dependence of the implied constant on $p$ is
explicit and polynomial. Here $\pi$ is a Hecke-Maass cusp form for $\rm SL_3(\mathbb{Z})$, $f$
is a holomorphic or Maass cusp form
for $\rm SL_2(\mathbb{Z})$, and $\chi$ is 
a primitive Dirichlet character of prime power conductor
$\mathfrak{q}=p^{\kappa}$ with $\kappa$ large enough.
This can be compared with the recent $t$-aspect subconvexity bound for $\rm GL_3\times GL_2$
$L$-functions (see Lin and the author \cite{LS})
\bna
L(1/2+it, \pi\otimes f)\ll_{\pi,f,\varepsilon}(1+|t|)^{3/2-3/20+\varepsilon}.
\ena
and will appear in a separate paper.

\end{remark}

\section{Sketch of the proof}

By the functional equation we have
$L\left(1/2,f\otimes g \otimes \chi\right)
\ll N^{-1/2}S(N)$, where
\bna
S(N)=\sum_{n\sim N}\lambda_f(n)\lambda_g(n)\mu\chi(n),
\ena
with $N\sim \mathfrak{q}^2$. Applying the conductor lowering mechanism
introduced by Munshi \cite{Munshi 1}, we have
\bna
S(N)=\mathop{\sum\sum}_{n,m\sim N\atop n\equiv
m ({\rm mod}\,p^{\lambda})}\lambda_g(n)\lambda_f(m)\chi(m)\delta\left(\frac{n-m}{p^\lambda}\right)
\ena
where $\delta: \mathbb{Z}\rightarrow \{0,1\}$ with
$\delta(0)=1$ and $\delta(n)=0$ for $n\neq 0$, and $\lambda\geq 2$ is an integer to be chosen later.
Using Duke, Friedlander and Iwaniec's delta method and removing the congruence
$n\equiv
m ({\rm mod}\,p^{\lambda})$ by exponential sums we get
\bna
S(N)\approx \frac{1}{Qp^{\lambda}}\sum_{q\sim Q\atop (q,p)=1}\frac{1}{q}\;
   \sideset{}{^*}\sum\limits_{a(\text{mod} \,qp^{\lambda})}\mathop{\sum\sum}_{n,m\sim N}
   \lambda_g(n)\lambda_f(m)\chi(m)e\left(\frac{a(m-n)}{qp^\lambda}\right),
\ena
where here and throughout, the $*$ on the sum ${\small \sum_{a(\bmod q)}}$ indicates
that the sum over $a$ is restricted to $(a,q)=1$, and we take
$Q=\sqrt{N/p^{\lambda}}$.
Trivially we have $S(N)\ll N^2$.

Recall $\chi$ is of modulus $\mathfrak{q}=p^{\kappa}$.
Then the conductor of the $m$-sum has the size $Qp^{\kappa}$. Applying $\rm GL_2$ Voronoi summation
to the $m$-sum we get that the dual sum is of size $Q^2p^{2\kappa}/N\asymp p^{2\kappa-\lambda}$.
The conductor for the $n$-sum has the size $Qp^\lambda$ and the dual sum after
$\rm GL_2$ Voronoi summation
is essentially supported on summation of size $Q^2p^{2\lambda}/N\asymp p^{\lambda}$. Because of
a Ramanujan sum appearing in the character sum and assuming
square-root cancellation for the remaining part of the character sum,
we find that we have saved
\bna
\frac{N}{Qp^{\kappa}}\times \frac{N}{Qp^{\lambda}}\times Q p^{\lambda/2}
\sim N.
\ena
So we are at the threshold and need to save a little more.

Now we arrive at an expression of the form
\bna
\sum_{q\sim Q}\sum_{m\sim p^{2\kappa-\lambda}}\lambda_f(m)
\sum_{n\sim p^{\lambda}\atop n\equiv -p^{\lambda}\overline{p^{2\kappa-\lambda}}m (\bmod\, q)}\quad
\sum_{b({\rm mod} p^{\lambda})}\lambda_f(n)
C^*(m,n,q),
\ena
where
\bna
\mathfrak{C}^*(m,n,q)=\sum_{c(\text{{\rm mod }} p^\kappa)}\overline{\chi}(c)
  \sideset{}{^*}\sum_{b (\text{{\rm mod }} p^{\lambda})
  }
  e\left(\frac{m\overline{q}\overline{(
  bp^{\kappa-\lambda}+cq)}}{p^{\kappa}}+
\frac{n\overline{q}\overline{b}}{p^\lambda}\right).
\ena
Next we apply Cauchy-Schwartz inequality to get rid of the Fourier coefficients.
Then we need to deal with
\bna
\sum_{q\sim Q}\sum_{m\sim p^{2\kappa-\lambda}}\left|
\sum_{n\sim p^{\lambda}\atop n\equiv -p^{\lambda}\overline{p^{2\kappa-\lambda}}m (\bmod\, q)}\quad
\sum_{b({\rm mod} p^{\lambda})}\lambda_f(n)
C^*(m,n,q)\right|^2.
\ena
Opening the square and applying Poisson summation to the
sum over $m$, we are able to save $p^{\lambda}/Q\sim p^{3\lambda/2-\kappa}$ from the diagonal term
and $p^{2\kappa-\lambda}/(Qp^{\kappa/2})\sim p^{(\kappa-\lambda)/2}$
from the off-diagonal term. So the optimal choice for $\lambda$
is given by $\lambda=3\kappa/4$. In total, we have saved
$
N\times \mathfrak{q}^{1/16}
$
It follows that
\bna
L\left(\frac{1}{2},f\otimes g \otimes \chi\right)\ll
N^{-1/2}S(N)\ll N^{1/2}\mathfrak{q}^{-1/16}\sim \mathfrak{q}^{1-1/16}.
\ena

\section{Proof of Theorem 1}

By the approximate functional equation we have
\bea\label{L-estimate}
L\left(\frac{1}{2},f\otimes g \otimes \chi\right)\ll_{f,g,\varepsilon} \mathfrak{q}^{\varepsilon}
\sup_{N\leq \mathfrak{q}^{2+\varepsilon}}\frac{|S(N)|}{\sqrt{N}},
\eea
where
\bna
S(N)=\sum_{n}\lambda_f(n)\lambda_g(n)\chi(n)
V\left(\frac{n}{N}\right)
\ena
for some smooth function $V$ supported in $[1,2]$ and satisfying $V^{(j)}(x)\ll_j 1$.
Note that by Cauchy-Schwartz inequality and the Rankin-Selberg estimate (see \cite{Mol})
\bea\label{RS}
\sum_{n\leq  x} \left|\lambda_f(n)\right|^2\ll_{f,\varepsilon}
x^{1+\varepsilon},
\eea
we have the trivial bound $S(N)\ll_{f,g,\varepsilon}N$.
%Thus Theorem
%1 is true for
%$N\leq \mathfrak{q}^{15/8}$.
%In the following, we will estimate $S(N)$ in the range
%\bea
%\mathfrak{q}^{15/8}<N\leq \mathfrak{q}^{2+\varepsilon}.
%\eea

\begin{proposition}
Assume $(2\kappa+1)/3\leq \lambda\leq 3\kappa/4$. Then we have
\bna
S(N)\ll N^{3/4+\varepsilon}\big(p^{\kappa-3\lambda/4}+p^{(\kappa+\lambda+3)/4}\big).
\ena
\end{proposition}

Take $\lambda=\lfloor 3\kappa/4\rfloor$, where $\lfloor x\rfloor$ denotes
the largest integer which does not exceed $x$.
By Proposition 1, we have
\bna
\mathscr{S}(N)\ll p^{3/4}N^{1/2+\varepsilon}\mathfrak{q}^{15/16}
\ena
from which Theorem 1 follows. The rest of the paper is devoted to prove Proposition 1.

\subsection{The circle method}
Define $\delta: \mathbb{Z}\rightarrow \{0,1\}$ with
$\delta(0)=1$ and $\delta(n)=0$ for $n\neq 0$.
By Duke, Friedlander and Iwaniec's delta method (see \cite[Chapter 20]{IK}),
we have
\bea\label{DFI's}
\delta(n)=\frac{1}{Q}\sum_{q\leq Q} \;\frac{1}{q}\;
\sideset{}{^*}\sum_{a\bmod{q}}e\left(\frac{na}{q}\right)
\int_\mathbb{R}g(q,\zeta) e\left(\frac{n\zeta}{qQ}\right)\mathrm{d}\zeta
\eea
where the function $g$ has the following properties
(see (20.158) and (20.159) of \cite{IK})
\bea\label{rapid decay g}
\frac{\partial^j}{\partial \zeta^j}g(q,\zeta)\ll Q^{\varepsilon}|\zeta|^{-j}, \qquad j\geq 0,
\eea
which implies that the effective range of the integration in
\eqref{DFI's} is $[-Q^\varepsilon, Q^\varepsilon]$.

\begin{lemma}\label{circle method} Suppose $\lambda\geq 1$.
Let $U\in C_c^{\infty}(-2,2)$ be a smooth positive function with
$U(x) = 1$ if $x\in [-1,1]$ and satisfying $U^{(j)}(x)\ll_j 1$. Then
\bna
\delta(n)
&=&\sum_{s=0}^{\lambda}\frac{1}{Q}\sum_{q\leq Q\atop (q,p)=1}\frac{1}{qp^\lambda}
\;\sideset{}{^*}\sum_{a(\text{{\rm mod }} qp^{\lambda-s})}
e\left(\frac{an}{qp^{\lambda-s}}\right)
\int_{\mathbb{R}}U\left(\frac{\zeta}{\mathfrak{q}^\varepsilon}\right)
  g(q,\zeta)e\left(\frac{n\zeta}{Qqp^\lambda}\right)\mathrm{d}\zeta\\
&&+\sum_{t=1}^{[\log Q/\log p]}
\frac{1}{Q}\sum_{q\leq Q/p^t\atop (q,p)=1}\frac{1}{qp^{\lambda+t}}
\sideset{}{^*}\sum_{a(\text{{\rm mod }} qp^{\lambda+t})}
e\left(\frac{an}{qp^{\lambda+t}}\right)
\int_{\mathbb{R}}U\left(\frac{\zeta}{\mathfrak{q}^\varepsilon}\right)
  g(p^tq,\zeta)e\left(\frac{n\zeta}{Qqp^{\lambda+t}}\right)\mathrm{d}\zeta
  +O_A\left(\mathfrak{q}^{-A}\right)
\ena
for any $A>0$.
\end{lemma}

\begin{proof}
By \eqref{DFI's} and \eqref{rapid decay g}, we have
\bna
\delta(n)=\frac{1}{Q}\sum_{q\leq Q} \;\frac{1}{q}\;
\sideset{}{^*}\sum_{a\bmod{q}}e\left(\frac{na}{q}\right)
\int_\mathbb{R}U\left(\frac{\zeta}{\mathfrak{q}^\varepsilon}\right)
g(q,\zeta) e\left(\frac{n\zeta}{qQ}\right)\mathrm{d}\zeta+O_A(\mathfrak{q}^{-A})
\ena
for any $A>0$.
Define $\mathbf{1}_\mathscr{F}=1$ if $\mathscr{F}$ is true, and is 0 otherwise.
Following Munshi \cite{Munshi 1} we write $\delta(n)$ as
$\delta(n/p^\lambda)\mathbf{1}_{p^\lambda|n}$ and detect the congruence by additive
characters to get
\bna
\delta(n)&=&\frac{1}{Q}\sum_{q\leq Q}\frac{1}{qp^\lambda}
\sum_{b(\text{{\rm mod }} p^\lambda)}\;
\sideset{}{^*}\sum_{a(\text{{\rm mod }} q)}e\left(\frac{a+bq}{qp^\lambda}n\right)
\int_{\mathbb{R}}U\left(\frac{\zeta}{\mathfrak{q}^\varepsilon}\right)
  g(q,\zeta)e\left(\frac{n\zeta}{Qqp^\lambda}\right)\mathrm{d}\zeta
\ena
which can be further written as $\delta_1(n)+\delta_2(n)$ with
\bna
\delta_1(n)&=&\frac{1}{Q}\sum_{q\leq Q\atop (q,p)=1}\frac{1}{qp^\lambda}
\sum_{b(\text{{\rm mod }} p^\lambda)}\;
\sideset{}{^*}\sum_{a(\text{{\rm mod }} q)}e\left(\frac{ap^{\lambda}+bq}{qp^\lambda}n\right)
\int_{\mathbb{R}}U\left(\frac{\zeta}{\mathfrak{q}^\varepsilon}\right)
  g(q,\zeta)e\left(\frac{n\zeta}{Qqp^\lambda}\right)\mathrm{d}\zeta,\\
\delta_2(n)&=&\frac{1}{Q}\sum_{q\leq Q/p}\frac{1}{qp^{\lambda+1}}
\sum_{b(\text{{\rm mod }} p^\lambda)}\;
\sideset{}{^*}\sum_{a(\text{{\rm mod }} qp)}e\left(\frac{a+bpq}{qp^{\lambda+1}}n\right)
\int_{\mathbb{R}}U\left(\frac{\zeta}{\mathfrak{q}^\varepsilon}\right)
  g(pq,\zeta)e\left(\frac{n\zeta}{Qqp^{\lambda+1}}\right)\mathrm{d}\zeta.
\ena
For $\delta_1(n)$, we have
\bna
\delta_1(n)
&=&\frac{1}{Q}\sum_{q\leq Q\atop (q,p)=1}\frac{1}{qp^\lambda}\;
\sideset{}{^*}\sum_{a(\text{{\rm mod }} qp^{\lambda})}
e\left(\frac{an}{qp^\lambda}\right)
\int_{\mathbb{R}}U\left(\frac{\zeta}{\mathfrak{q}^\varepsilon}\right)
  g(q,\zeta)e\left(\frac{n\zeta}{Qqp^\lambda}\right)\mathrm{d}\zeta\\
&&+\frac{1}{Q}\sum_{q\leq Q\atop (q,p)=1}\frac{1}{qp^\lambda}
\sum_{b(\text{{\rm mod }} p^{\lambda-1})}\;
\sideset{}{^*}\sum_{a(\text{{\rm mod }} q)}
e\left(\frac{ap^{\lambda-1}+bq}{qp^{\lambda-1}}n\right)
\int_{\mathbb{R}}U\left(\frac{\zeta}{\mathfrak{q}^\varepsilon}\right)
  g(q,\zeta)e\left(\frac{n\zeta}{Qqp^\lambda}\right)\mathrm{d}\zeta\\
&=&\sum_{s=0}^{\lambda}\frac{1}{Q}\sum_{q\leq Q\atop (q,p)=1}\frac{1}{qp^\lambda}
\;\sideset{}{^*}\sum_{a(\text{{\rm mod }} qp^{\lambda-s})}
e\left(\frac{na}{qp^{\lambda-s}}\right)
\int_{\mathbb{R}}U\left(\frac{\zeta}{\mathfrak{q}^\varepsilon}\right)
  g(q,\zeta)e\left(\frac{n\zeta}{Qqp^\lambda}\right)\mathrm{d}\zeta.
\ena
For $\delta_2(n)$, we have
\bna
\delta_2(n)
&=&\frac{1}{Q}\sum_{q\leq Q/p\atop (q,p)=1}\frac{1}{qp^{\lambda+1}}
\sum_{b(\text{{\rm mod }} p^\lambda)}\;
\sideset{}{^*}\sum_{a(\text{{\rm mod }} qp)}e\left(\frac{a+bpq}{qp^{\lambda+1}}n\right)
\int_{\mathbb{R}}U\left(\frac{\zeta}{\mathfrak{q}^\varepsilon}\right)
  g(pq,\zeta)e\left(\frac{n\zeta}{Qqp^{\lambda+1}}\right)\mathrm{d}\zeta\\
&&+\frac{1}{Q}\sum_{q\leq Q/p^2}\frac{1}{qp^{\lambda+2}}
\sum_{b(\text{{\rm mod }} p^\lambda)}\;
\sideset{}{^*}\sum_{a(\text{{\rm mod }} qp^2)}e\left(\frac{a+bp^2q}{qp^{\lambda+2}}n\right)
\int_{\mathbb{R}}U\left(\frac{\zeta}{\mathfrak{q}^\varepsilon}\right)
  g(p^2q,\zeta)e\left(\frac{n\zeta}{Qqp^{\lambda+2}}\right)\mathrm{d}\zeta\\
&=&\frac{1}{Q}\sum_{q\leq Q/p\atop (q,p)=1}\frac{1}{qp^{\lambda+1}}
\sideset{}{^*}\sum_{a(\text{{\rm mod }} qp^{\lambda+1})}
e\left(\frac{an}{qp^{\lambda+1}}\right)
\int_{\mathbb{R}}U\left(\frac{\zeta}{\mathfrak{q}^\varepsilon}\right)
  g(pq,\zeta)e\left(\frac{n\zeta}{Qqp^{\lambda+1}}\right)\mathrm{d}\zeta\\
&&+\frac{1}{Q}\sum_{q\leq Q/p^2}\frac{1}{qp^{\lambda+2}}
\sum_{b(\text{{\rm mod }} p^\lambda)}\;
\sideset{}{^*}\sum_{a(\text{{\rm mod }} qp^2)}e\left(n\frac{a+bp^2q}{qp^{\lambda+2}}\right)
\int_{\mathbb{R}}U\left(\frac{\zeta}{\mathfrak{q}^\varepsilon}\right)
  g(p^2q,\zeta)e\left(\frac{n\zeta}{Qqp^{\lambda+2}}\right)\mathrm{d}\zeta\\
&=&\sum_{t=1}^{[\log Q/\log p]}
\frac{1}{Q}\sum_{q\leq Q/p^t\atop (q,p)=1}\frac{1}{qp^{\lambda+t}}
\sideset{}{^*}\sum_{a(\text{{\rm mod }} qp^{\lambda+t})}
e\left(\frac{an}{qp^{\lambda+t}}\right)
\int_{\mathbb{R}}U\left(\frac{\zeta}{\mathfrak{q}^\varepsilon}\right)
  g(p^tq,\zeta)e\left(\frac{n\zeta}{Qqp^{\lambda+t}}\right)\mathrm{d}\zeta
  +O_A\left(\mathfrak{q}^{-A}\right).
\ena
This proves the lemma.
\end{proof}

Now we write
\bna
S(N)=\sum_{n}\lambda_g(n)W\left(\frac{n}{N}\right)
\sum_{p^\lambda|m-n}
\lambda_f(m)\chi(m)V\left(\frac{m}{N}\right)
\delta\left(\frac{n-m}{p^\lambda}\right),
\ena
where $W$ is a smooth function supported in $[1/2,5/2]$, $W(x)=1$
for $x\in [1,2]$ and $x^{(j)}(x)\ll_j 1$.
Applying Lemma \ref{circle method} with
\bea\label{Q}
Q=\sqrt{N/p^{\lambda}}
\eea
with $\lambda\in \mathbb{N}$ ($2\leq \lambda<\kappa$) being a parameter to be determined later,
we have
\bna
S(N)\ll \mathfrak{q}^{\varepsilon}|S^\flat(N)|,
\ena
where
\bna
S^\flat(N)
&=&\sum_n\lambda_g(n)W\left(\frac{n}{N}\right)
\sum_{m}\lambda_f(m)\chi(m)V\left(\frac{m}{N}\right)\\
&&\times \frac{1}{Q}\sum_{q\leq Q\atop (q,p)=1}\frac{1}{qp^\lambda}
\;\sideset{}{^*}\sum_{a(\text{{\rm mod }} qp^{\lambda})}
e\left(\frac{a(m-n)}{qp^{\lambda}}\right)
\int_{\mathbb{R}}U\left(\frac{\zeta}{\mathfrak{q}^\varepsilon}\right)
  g(q,\zeta)e\left(\frac{(m-n)\zeta}{Qqp^\lambda}\right)\mathrm{d}\zeta.
 \ena
Therefore, our task is to prove the estimate in Proposition 1 for $S^\flat(N)$.
We rearrange $S^\flat(N)$ as
\bea\label{S}
S^\flat(N)
=\frac{1}{Qp^\lambda}\int_{\mathbb{R}}U\left(\frac{\zeta}{\mathfrak{q}^\varepsilon }\right)
  \sum_{q\leq Q\atop (q,p)=1}\frac{g(q,\zeta)}{q}
  \sideset{}{^*}\sum_{a(\text{{\rm mod }} qp^\lambda)}
  \mathscr{A}\times \mathscr{B} \hspace{3pt} \mathrm{d}\zeta,
\eea
where
\bea\label{A}
\mathscr{A}=\sum_m\lambda_f(m)\chi(m)e\left(\frac{am}{qp^\lambda}\right)
 V\left(\frac{m}{N}\right)e\left(\frac{m\zeta}{Qqp^\lambda}\right)
\eea
and
\bea\label{B}
\mathscr{B}=\sum_n\lambda_g(n)e\left(-\frac{an}{qp^\lambda}\right)
 W\left(\frac{n}{N}\right)e\left(\frac{-n\zeta}{Qqp^\lambda}\right).
\eea
\subsection{Voronoi summation}

Next we transform $\mathscr{A}$ and $\mathscr{B}$ by $\rm GL_2$
Voronoi summation formula and obtain the following results.

\begin{lemma}\label{A-voronoi}
Let $a^*=ap^{\kappa-\lambda}+cq$. We have
\bna
\mathscr{A}=
\frac{ N^{1/2}}{\tau(\overline{\chi})}
\sum_{c(\text{{\rm mod }} p^\kappa)}\overline{\chi}(c)
\sum_\pm\sum_{m\leq p^{2\kappa-\lambda+\varepsilon}}
\frac{\lambda_f(m)}{m^{1/2}}e\left(\pm\frac{m\overline{a^*}}{qp^\kappa}\right)
\mathfrak{I}^\pm\left(\frac{\pi^2m}{q^2p^{2\kappa}},q,\zeta\right)+O_A\left(\mathfrak{q}^{-A}\right)
\ena
for any $A>0$, where $\mathfrak{I}^{\pm}\left(x,q,\zeta\right)$ is
defined in \eqref{GL2 integral-1}.
\end{lemma}

\begin{lemma}\label{B-voronoi}
We have
\bna
\mathscr{B}=N^{1/2}\sum_\pm
\sum_{n\leq p^{\lambda+\epsilon}}\frac{\lambda_g(n)}{n^{1/2}}
e\left(\mp\frac{n\overline{a}}{qp^\lambda}\right)
\mathfrak{J}^\pm\left(\frac{\pi^2n}{q^2p^{2\lambda}},q,\zeta\right)+O_A(\mathfrak{q}^{-A})
\ena
for any $A>0$, where $\mathfrak{J}^{\pm}\left(y,q,\zeta\right)$
is defined in \eqref{GL2 integral-3}.
\end{lemma}

The details of the proof of Lemmas \ref{A-voronoi} and \ref{B-voronoi} are in Sections 4 and 5.
Plugging Lemmas 2 and 3 into \eqref{S}, we have
\bea\label{S-flat(N)}
S^{\flat}(N)&=&\sum_{\pm}\sum_{\pm}
  \frac{N}{Qp^{\lambda}\tau(\overline{\chi})}
  \sum_{q\leq Q\atop (q,p)=1}\;\frac{1}{q}
 \sum_{m\leq p^{2\kappa-\lambda+\varepsilon}}
  \frac{\lambda_f(m)}{m^{1/2}}
  \sum_{n\leq p^{\lambda+\epsilon}}\frac{\lambda_g(n)}{n^{1/2}}
 \nonumber\\
&&\times \mathfrak{C}(\pm m,\mp n,q)\mathfrak{R}\left(\frac{\pi^2m}{q^2p^{2\kappa}},
\frac{\pi^2n}{q^2p^{2\lambda}},q\right)+O_A\left(\mathfrak{q}^{-A}\right),
\eea
where
\bea\label{character sum1}
\mathfrak{C}(m,n,q)=
\sideset{}{^*}\sum_{a(\text{{\rm mod }} qp^\lambda)}
  \sum_{c(\text{{\rm mod }} p^\kappa)}
  \overline{\chi}(c)
  e\left(\frac{m\overline{(ap^{\kappa-\lambda}+cq)}}{qp^\kappa}
  +\frac{n\overline{a}}{qp^\lambda}\right)
\eea
and
\bea\label{integral 4}
\mathfrak{R}(x_1,x_2,q)=
\int_{\mathbb{R}}
  g(q,\zeta)U\left(\frac{\zeta}{\mathfrak{q}^\varepsilon}\right)\mathfrak{I}^\pm\left(x_1,q,\zeta\right)
  \mathfrak{J}^{\pm}\left(x_2,q,\zeta\right)
  \mathrm{d}\zeta.
\eea
Note that
\bna
e\left(\frac{m\overline{(ap^{\kappa-\lambda}+cq)}}{qp^{\kappa}}\right)=
e\left(\frac{m\overline{q}\overline{(ap^{\kappa-\lambda}+cq)}}{p^{\kappa}}\right)
e\left(\frac{m\overline{ap^{2\kappa-\lambda}}}{q}\right).
\ena
The character sum in \eqref{character sum1} is
\bea\label{character sum2}
\mathfrak{C}(m,n,q)
&=&\sum_{c(\text{{\rm mod }} p^\kappa)}\overline{\chi}(c)
  \sideset{}{^*}\sum_{b (\text{{\rm mod }} p^{\lambda})}
  e\left(\frac{m\overline{q}\overline{(
  bp^{\kappa-\lambda}+cq)}}{p^{\kappa}}+
\frac{n\overline{q}\overline{b}}{p^\lambda}\right)
\sideset{}{^*}\sum_{a(\text{{\rm mod }} q)}\;
  e\left(\frac{n\overline{p^{\lambda}}+m\overline{p^{2\kappa-\lambda}}}{q}\overline{a}\right)\nonumber\\
&=&\sum_{c(\text{{\rm mod }} p^\kappa)}\overline{\chi}(c)
  \sideset{}{^*}\sum_{b (\text{{\rm mod }} p^{\lambda})
  }
  e\left(\frac{m\overline{q}\overline{(
  bp^{\kappa-\lambda}+cq)}}{p^{\kappa}}+
\frac{n\overline{q}\overline{b}}{p^\lambda}\right)
\sum_{d|q\atop n\equiv-mp^{\lambda}\overline{p^{2\kappa-\lambda}}\bmod d}d
\mu\left(\frac{q}{d}\right).
\eea

Plugging \eqref{character sum2} into \eqref{S-flat(N)} and
reducing the $m,n$ sums in $S^\flat(N)$ into dyadic intervals, we have
\bea\label{dyadic-0}
S^\flat(N)\ll \frac{N^{1/2+\varepsilon}}{p^{(\kappa+\lambda)/2}}\sum_{\pm}\sum_{\pm}
\max_{1\ll M\ll p^{2\kappa-\lambda+\varepsilon} \atop M \mathrm{dyadic}}
\max_{1\ll N_1\ll p^{\lambda+\varepsilon}\atop N_1 \mathrm{dyadic}}
\sum_{q\leq Q\atop (q,p)=1}
\sum_{d|q}q^{-1}d\;
|S^\flat(M,N_1,q,d,\pm,\pm)|
+\mathfrak{q}^{-1},
\eea
where
\bea\label{S1}
S^\flat(M,N_1,q,d,\pm,\pm)
=\sum_{m\sim M}\frac{\lambda_f(m)}{m^{1/2}}
\sum_{n\sim N_1\atop n\equiv-mp^{\lambda}\overline{p^{2\kappa-\lambda}}\bmod d}
\frac{\lambda_g(n)}{n^{1/2}}
 \mathfrak{C}^*(\pm m,\mp n,q)\mathfrak{R}\left(\frac{\pi^2m}{q^2p^{2\kappa}},
\frac{\pi^2n}{q^2p^{2\lambda}},q\right)
\eea
with
\bea\label{character sum3}
\mathfrak{C}^*(m,n,q)=\sum_{c(\text{{\rm mod }} p^\kappa)}\overline{\chi}(c)
  \sideset{}{^*}\sum_{b (\text{{\rm mod }} p^{\lambda})
  }
  e\left(\frac{m\overline{q}\overline{(
  bp^{\kappa-\lambda}+cq)}}{p^{\kappa}}+
\frac{n\overline{q}\overline{b}}{p^\lambda}\right).
\eea

Before further analysis, we give an estimate for the integral $\mathfrak{R}(x_1,x_2,q)$.
By \eqref{GL2 integral-1} and \eqref{GL2 integral-3}, we have
\bna
\mathfrak{R}(x_1,x_2,q)&=&\frac{1}{4}
\int_{\mathbb{R}}
  g(q,\zeta)U\left(\frac{\zeta}{\mathfrak{q}^\varepsilon}\right)
\int_{\mathbb{R}}(Nx)^{-i\tau_1}
\rho_f^{\pm}\left(-\frac{1}{2}+i\tau_1\right)
V^{\dag}\left(-\frac{\zeta Q}{q},\frac{1}{2}-i\tau_1\right)\\&&\times
\int_{\mathbb{R}}(Nx)^{-i\tau_2}
\rho_g^{\pm}\left(-\frac{1}{2}+i\tau_2\right)
W^{\dag}\left(\frac{\zeta Q}{q},\frac{1}{2}-i\tau_2\right)\mathrm{d}\tau_1\mathrm{d}\tau_2
  \mathrm{d}\zeta.
\ena
Using the estimate
$W^{\dag}(r,\sigma+i\tau)\ll\min\left\{1, \left(\frac{1+|r|}{|\tau|}\right)^j\right\}$,
the integral over $\tau_1$ and $\tau_2$ can be restricted in
$|\tau_1|,|\tau_2|\leq N^{\varepsilon}Q/q$, up to a negligible error.
By applying smooth partitions of unity to the variable $|\tau_1|$ and $|\tau_2|$, we have
\bna
\mathfrak{R}(x_1,x_2,q)&=&
\sum_{N^{\varepsilon}\ll \Xi_1,\Xi_2\ll N^{\varepsilon}Q/q \atop \Xi_1,\Xi_2 \,\mathrm{dyadic}}
\frac{1}{4}
\int_{\mathbb{R}}\int_{\mathbb{R}}\omega\left(\frac{|\tau_1|}{\Xi_1}\right)
\omega\left(\frac{|\tau_2|}{\Xi_2}\right)
\rho_f^{\pm}\left(-\frac{1}{2}+i\tau_1\right)\rho_g^{\pm}\left(-\frac{1}{2}+i\tau_2\right)
\nonumber\\&&\times
(Nx_1)^{-i\tau_1}(Nx_2)^{-i\tau_2}
\int_{\mathbb{R}}
  g(q,\zeta)U\left(\frac{\zeta}{\mathfrak{q}^\varepsilon}\right)
V^{\dag}\left(-\frac{\zeta Q}{q},\frac{1}{2}-i\tau_1\right)\nonumber\\&&\qquad\times
W^{\dag}\left(\frac{\zeta Q}{q},\frac{1}{2}-i\tau_2\right)  \mathrm{d}\zeta
\mathrm{d}\tau_1\mathrm{d}\tau_2+O(N^{\varepsilon}),
\ena
where $\omega\in C_c^{\infty}(1,2)$ and $\omega^{(j)}(x)\ll_j 1$.
By a stationary phase analysis (see for example, \cite[Lemma 3.1]{KPY}), we have
\bna
V^{\dag}\left(-\frac{\zeta Q}{q},\frac{1}{2}-i\tau_1\right)=\frac{1}{\sqrt{\tau_1}}
V_{\sharp}\left(\frac{q\tau_1}{2\pi \zeta Q}\right)
\left(\frac{q\tau_1}{2\pi e\zeta Q}\right)^{-i\tau_1}+O(N^{-100})
\ena
and
\bna
W^{\dag}\left(\frac{\zeta Q}{q},\frac{1}{2}-i\tau_2\right) =\frac{1}{\sqrt{\tau_2}}
W^{\sharp}\left(\frac{-q\tau_2}{2\pi \zeta Q}\right)
\left(\frac{-q\tau_2}{2\pi e\zeta Q}\right)^{-i\tau_2}+O(N^{-100}),
\ena
where $V_{\sharp}(x)$ and $W^{\sharp}(x)$ are 1-inert functions supported on $x\asymp 1$.
Thus
\bna
\mathfrak{R}(x_1,x_2,q)&=&
\sum_{N^{\varepsilon}\ll \Xi_1,\Xi_2\ll N^{\varepsilon}Q/q \atop \Xi_1,\Xi_2 \,\mathrm{dyadic}}
\frac{1}{4}
\int_{\mathbb{R}}\int_{\mathbb{R}}
\frac{1}{\sqrt{\tau_1\tau_2}}\omega\left(\frac{|\tau_1|}{\Xi_1}\right)
\omega\left(\frac{|\tau_2|}{\Xi_2}\right)
\rho_f^{\pm}\left(-\frac{1}{2}+i\tau_1\right)\rho_g^{\pm}\left(-\frac{1}{2}+i\tau_2\right)
\nonumber\\&&\times\mathcal{G}(\tau_1,\tau_2)
\left(\frac{q\tau_1}{2\pi e Q}\right)^{-i\tau_1}
\left(\frac{-q\tau_2}{2\pi e Q}\right)^{-i\tau_2}
(Nx_1)^{-i\tau_1}(Nx_2)^{-i\tau_2}
\mathrm{d}\tau_1\mathrm{d}\tau_2+O(N^{\varepsilon})
\ena
with
\bna
\mathcal{G}(\tau_1,\tau_2)=\int_{\mathbb{R}}
  g(q,\zeta)U\left(\frac{\zeta}{\mathfrak{q}^\varepsilon}\right)
V_{\sharp}\left(\frac{q\tau_1}{2\pi \zeta Q}\right)
W^{\sharp}\left(\frac{-q\tau_2}{2\pi \zeta Q}\right)
\zeta^{i(\tau_1+\tau_2)}
\mathrm{d}\zeta.
\ena
For the integral $\mathcal{G}(\tau_1,\tau_2)$, repeated integration by parts shows that the contribution from
$|\tau_1+\tau_2|\geq N^{\varepsilon}$ is arbitrarily small. Therefore,
\bea\label{integral estimate}
\mathfrak{R}(x_1,x_2,q)\ll
\sum_{N^{\varepsilon}\ll \Xi_1,\Xi_2\ll N^{\varepsilon}Q/q \atop \Xi_1,\Xi_2 \,\mathrm{dyadic}}
\int\limits_{|\tau_1|\asymp \Xi_1}\int\limits_{|\tau_2|\asymp \Xi_2\atop |\tau_1+\tau_2|\leq N^{\varepsilon}}
\frac{1}{\sqrt{\tau_1\tau_2}}
\mathrm{d}\tau_1\mathrm{d}\tau_2+N^{\varepsilon}\ll N^{\varepsilon}.
\eea

\subsection{Cauchy-Schwartz and Poisson summation}
Applying the Cauchy-Schwartz inequality to the $m$-sum in \eqref{S1} and
using the Rankin-Selberg estimate in \eqref{RS}, we have
\bea\label{to T}
S^\flat(M,N_1,q,d,\pm,\pm)\ll \mathbf{T}^{1/2},
\eea
where, temporarily,
\bea\label{T}
\mathbf{T}
=\sum_{m}\varpi\left(\frac{m}{M}\right)
  \bigg|\sum_{n\sim N_1 \atop n\equiv-mp^{\lambda}\overline{p^{2\kappa-\lambda}}\bmod d}
  \frac{\lambda_g(n)}{n^{1/2}}
 \mathfrak{C}^*(\pm m,\mp n,q)\mathfrak{R}\left(\frac{\pi^2m}{q^2p^{2\kappa}},
\frac{\pi^2n}{q^2p^{2\lambda}},q\right)\bigg|^2.
\eea
Here $\varpi(x)$ is a smooth nonnagative function, $\varpi(x)=1$
if $x\in[1,2]$ and $\varpi^{(j)}(x)\ll_j 1$.

Opening the square in \eqref{T} and switching the order of summations ,we get
\bea\label{T-estimate}
\mathbf{T}
=\sum_{n_1\sim N_1 }
  \frac{\lambda_g(n_1)}{n_1^{1/2}}
  \sum_{n_2\sim N_1 \atop n_2\equiv n_1\bmod d}
  \frac{\overline{\lambda_g(n_2)}}{n_2^{1/2}}
\times\mathscr{H}
\eea
where
\bna
\mathscr{H}&=&\sum_{m\in \mathbb{Z}\atop m\equiv -n_1p^{2\kappa-\lambda}\overline{p^{\lambda}}\bmod d}
\varpi\left(\frac{m}{M}\right) \mathfrak{C}^*(\pm m,\mp n_1,q)
\overline{\mathfrak{C}^*(\pm m,\mp n_2,q)}\\
&&\times
\mathfrak{R}\left(\frac{\pi^2m}{q^2p^{2\kappa}},
\frac{\pi^2n_1}{q^2p^{2\lambda}},q\right)
\overline{\mathfrak{R}\left(\frac{\pi^2m}{q^2p^{2\kappa}},
\frac{\pi^2n_2}{q^2p^{2\lambda}},q\right)}.
\ena
Applying Poisson summation with modulus $dp^{\kappa}$,
we arrive at
\bna
\mathscr{H}=
\frac{M}{dp^{\kappa}}
\sum_{m\in \mathbb{Z}}\mathcal{C}\times \mathcal{K},
\ena
where
\bea\label{character sum}
\mathcal{C}=\sum_{\beta\bmod p^{\kappa}}
\mathfrak{C}^*(\pm \beta,\mp n_1,q)
\overline{\mathfrak{C}^*(\pm \beta,\mp n_2,q)}
e\left(\frac{m(-n_1p^{3\kappa-\lambda}\overline{p^{\kappa+\lambda}}+d\overline{d}\beta)}
{dp^{\kappa}}\right)
\eea
and
\bea\label{integral 5}
\mathcal{K}=\int_{\mathbb{R}}W(x)
\mathfrak{R}\left(\frac{\pi^2Mx}{q^2p^{2\kappa}},
\frac{\pi^2n_1}{q^2p^{2\lambda}},q\right)
\overline{\mathfrak{R}\left(\frac{\pi^2Mx}{q^2p^{2\kappa}},
\frac{\pi^2n_2}{q^2p^{2\lambda}},q\right)}
e\left(-\frac{m M x}{dp^{\kappa}}\right)\mathrm{d}x.
\eea
By integration by parts and using \eqref{integral estimate}, we have
\bna
\mathcal{K}&\ll&N^{\varepsilon}
\left(\frac{|m|M}{dp^{\kappa}}\right)^{-j}.
\ena
Thus the contribution from $|m|\geq N^{\varepsilon}dp^{\kappa}/M$ for any $\varepsilon>0$
is negligible. For smaller $m$, we use the trivial bound
\bea\label{K-estimate}
\mathcal{K}\ll N^{\varepsilon}.
\eea
For the character sum, we have the following estimate which will be proved in the last section. 

\begin{lemma}\label{character sum estimate}
Assume $(2\kappa+1)/3\leq \lambda\leq 3\kappa/4$. The character sum vanishes unless $(n_1n_2,p)=1$ and
$p^{\kappa-\lambda}|m$.
Let $m=m'p^{\kappa-\lambda}$, $\lambda=2\alpha+\delta_1$ and
$\kappa=2\beta+\delta_2$ with $\delta_1,\delta_2=0$ or 1, $\alpha\geq 1$ and $\beta\geq 1$.

(1) If $p^{\beta-\kappa+\lambda}|m'$, , then $n_1\equiv n_2 (\bmod \,p^{\beta-\kappa+\lambda})$ and
\bna
\mathcal{C}^*(n_1,n_2,m) \ll p^{2\kappa+\lambda+\delta_1}.
\ena

(2) If $p^{\ell}\| m'$ with $\ell<\beta-\kappa+\lambda$, then $p^{\ell}\| n_1-n_2$ and
\bna
\mathcal{C}^*(n_1,n_2,m) \ll p^{5\kappa/2+2\ell +\delta_1+\delta_2/2}.
\ena

(3) For $m=0$, we have
\bna
\mathcal{C}^*(n_1,n_2,m)=p^{2\kappa}\sum_{d|(n_1-n_2,p^{\lambda})}d\mu\left(p^{\lambda}/d\right).
\ena

\end{lemma}
By \eqref{K-estimate} and Lemma \ref{character sum estimate}, we get
\bea\label{H-estimate}
\mathscr{H}&&\ll
\frac{N^{\varepsilon}M}{dp^{\kappa}}\bigg(p^{2\kappa+\lambda}
\mathbf{1}_{n_1\equiv n_2\bmod p^{\lambda}}
+p^{2\kappa+\lambda-1}\mathbf{1}_{n_1\equiv n_2\bmod p^{\lambda-1}}\nonumber\\&&
+\sum_{0\neq |m'|\leq N^{\varepsilon}dp^{\lambda}Q/(qM)\atop p^{\beta-\kappa+\lambda}|m',n_1-n_2}
p^{2\kappa+\lambda+\delta_1}
+\sum_{1\leq \ell\leq \lambda-\kappa/2\atop p^{\ell}\|n_1-n_2}
\sum_{0\neq |m'|\leq N^{\varepsilon}dp^{\lambda}Q/(qM)\atop p^{\ell}\|m'}
p^{5\kappa/2+2\ell +\delta_1+\delta_2/2}
\bigg),
\eea
where $(2\kappa+1)/3\leq \lambda\leq 3\kappa/4$, $\lambda=2\alpha+\delta_1$ and
$\kappa=2\beta+\delta_2$ with $\delta_1,\delta_2=0$ or 1, $\alpha\geq 1$ and $\beta\geq 1$.
By plugging \eqref{H-estimate} into \eqref{T-estimate} and using the estimate
$\lambda_g(n_1)\overline{\lambda_g(n_2)}\ll |\lambda_g(n_1)|^2+|\lambda_g(n_2)|^2$, we have
\bna
\mathbf{T}
&\ll&
\frac{N^{\varepsilon}M}{dp^{\kappa}}
\sum_{n_1\sim N_1 }
  \frac{|\lambda_g(n_1)|^2}{n_1^{1/2}}\bigg(
  \sum_{n_2\sim N_1 \atop n_2\equiv n_1\bmod dp^{\lambda}}
  \frac{1}{n_2^{1/2}}p^{2\kappa+\lambda}
  +  \sum_{n_2\sim N_1 \atop n_2\equiv n_1\bmod dp^{\lambda-1}}
  \frac{1}{n_2^{1/2}}p^{2\kappa+\lambda-1}
  \bigg)\\
&&+\frac{N^{\varepsilon}M}{dp^{\kappa}}\sum_{\ell\ll 1}
\sum_{n_1\sim N_1 }
  \frac{|\lambda_g(n_1)|^2}{n_1^{1/2}}
  \sum_{n_2\sim N_1 \atop n_2\equiv n_1\bmod dp^{\beta-\kappa+\lambda}}
  \frac{1}{n_2^{1/2}}\frac{dQ}{qM}p^{5\kappa/2+\lambda+\ell +\delta_1+\delta_2/2}\\
&&+\frac{N^{\varepsilon}M}{dp^{\kappa}}\sum_{1\leq \ell\leq \lambda-\kappa/2}
\sum_{n_1\sim N_1 }
  \frac{|\lambda_g(n_1)|^2}{n_1^{1/2}}
  \sum_{n_2\sim N_1 \atop n_2\equiv n_1\bmod dp^{\ell}}
  \frac{1}{n_2^{1/2}}\frac{dQ}{qM}p^{5\kappa/2+\lambda+\ell +\delta_1+\delta_2/2}\bigg).
\ena
Notice that the last term dominates the second term. Therefore,
\bna
\mathbf{T}
&\ll&
\frac{N^{\varepsilon}M}{d}p^{\kappa+\lambda}
\bigg(1+\frac{N_1}{dp^{\lambda}}\bigg)+\frac{p^{3/2}N^{\varepsilon}M}{dp^{\kappa}}
\left(\frac{dQ}{qM}p^{\lambda-\kappa/2}+\frac{QN_1}{qM}\right)p^{3(\kappa+1)/2+\lambda}\\
&\ll&\frac{N^{\varepsilon}M}{d}\left(p^{\kappa+\lambda}
+\frac{Qp^{\lambda}}{qM}p^{3(\kappa+1)/2+\lambda}\right)\\
&\ll&\frac{N^{\varepsilon}}{d}\left(Mp^{\kappa+\lambda}
+\frac{Q}{q}p^{3(\kappa+1)/2+2\lambda}\right)\\
&\ll&\frac{N^{\varepsilon}}{d}\left(p^{3\kappa}
+\frac{Q}{q}p^{3(\kappa+1)/2+2\lambda}\right),
\ena
where recalling \eqref{dyadic-0} that $M\ll p^{2\kappa-\lambda+\varepsilon}$ and
$N_1\ll p^{\lambda+\varepsilon}$. Here we have used \eqref{RS}. Plugging this estimate into
\eqref{to T}, we have
\bea\label{to T-2}
S^\flat(M,N_1,q,d,\pm,\pm)\ll \frac{N^{\varepsilon}}{d^{1/2}}
\left(p^{3\kappa/2}+\frac{Q^{1/2}}{q^{1/2}}p^{3(\kappa+1)/4+\lambda}\right).
\eea

\subsection{Conclusion}

By \eqref{dyadic-0} and \eqref{to T-2}, one has
\bna
S^\flat(N)&\ll& \frac{N^{1/2+\varepsilon}}{p^{(\kappa+\lambda)/2}}
\sum_{q\leq Q\atop (q,p)=1}q^{-1}
\sum_{d|q}d^{1/2}\left(p^{3\kappa/2}+\frac{Q^{1/2}}{q^{1/2}}p^{3(\kappa+1)/4+\lambda}\right)\\
&\ll& N^{3/4+\varepsilon}\big(p^{\kappa-3\lambda/4}+p^{(\kappa+\lambda+3)/4}\big).
\ena
This completes the proof of Proposition 1.

\section{Proof of Lemma \ref{A-voronoi}}
\setcounter{equation}{0}
\medskip

In this section we apply $\rm GL_2$ Voronoi formula to transform $\mathscr{A}$ in \eqref{A}.
By the Fourier expansion of $\chi$ in the terms of additive characters,
we have
\bna
\mathscr{A}
=\frac{1}{\tau(\overline{\chi})}\sum_{c(\text{{\rm mod }} p^\kappa)}
\overline{\chi}(c)\sum_m\lambda_f(m)
e\left(\frac{ap^{\kappa-\lambda}+cq}{qp^\kappa}m\right)V\left(\frac{m}{N}\right)
e\left(\frac{m\zeta}{Qqp^\lambda}\right).
\ena
Note that $(ap^{\kappa-\lambda}+cq,q)=1$ and $(ap^{\kappa-\lambda}+cq,p)=(cq,p)=1$.
Thus $(ap^{\kappa-\lambda}+cq, qp^{\kappa})=1$. Denote $a^*=ap^{\kappa-\lambda}+cq$.
The $m$-sum is
\bna
\sum_m\lambda_f(m)e\left(\frac{a^*m}{qp^\kappa}\right)
V\left(\frac{m}{N}\right)e\left(\frac{m\zeta}{Qqp^\lambda}\right)
:=\sum_m\lambda_f(m)e\left(\frac{a^*m}{qp^\kappa}\right)\varphi\left(m\right),
\ena
where $\varphi(y)=V(y/N)e\left(\frac{\zeta y}{Qqp^\lambda}\right)$.

Let $f$ be a Hecke-Maass cusp form with Laplace eigenvalue $1/4+\mu_f^2$.
Without loss of generality, we assume $f$ is even. We have the following
Voronoi formula for $f$ (see \cite{MS}).

\begin{lemma}\label{GL2 Voronoi formula}
Let $\varphi(x)$ be a smooth function compactly supported on $\mathbb{R^+}$.
Let $a, \overline{a}, c\in\mathbb{Z}$ with $c\neq0, (a,c)=1$ and
$a\overline{a}\equiv1\;(\text{{\rm mod }} c)$. Then
\bea
\sum_m\lambda_f(m)e\left(\frac{am}{c}\right)\varphi(m)
=c\sum_\pm\sum_m\frac{\lambda_f(m)}{m}e\left(\pm\frac{\overline{a}m}{c}\right)
\Psi_\pm\left(\frac{m}{c^2}\right),
\eea
where for $\sigma>-1$,
\bea\label{GL2 integral-11}
\Psi_\pm(x)&=&\frac{1}{2\pi i}\int_{(\sigma)}(\pi^2 x)^{-s}\rho_f^{\pm}(s)\widetilde{\varphi}(-s) \mathrm{d}s,
\eea
with
\bna
\rho_f^{\pm}(s)=\frac{1}{2\pi}\left(\prod_\pm\frac{\Gamma(\frac{1+s\pm i\mu_f}{2})}
{\Gamma(\frac{-s\pm i\mu_f}{2})}\pm
\prod_\pm\frac{\Gamma(\frac{2+s\pm i\mu_f}{2})}
{\Gamma(\frac{1-s\pm i\mu_f}{2})}\right).
\ena
Here $\widetilde{\varphi}$ is the Mellin transform of $\varphi$.
\end{lemma}

Applying the $GL_2$ Voronoi formula in Lemma \ref{GL2 Voronoi formula}, we have
\bea
\sum_m\lambda_f(m)e\left(\frac{a^*m}{qp^\kappa}\right)\varphi\left(m\right)
=qp^\kappa\sum_\pm\sum_m\frac{\lambda_f(m)}{m}e\left(\pm\frac{\overline{a^*}m}{qp^\kappa}\right)
\Psi_\pm\left(\frac{m}{q^2p^{2\kappa}}\right),
\eea
where $\Psi_\pm(x)$ is defined in \eqref{GL2 integral-11}.
By Stirling's formula, for $\sigma\geq -1/2$,
\bna
\rho_f^{\pm}(\sigma+i\tau)\ll(1+|\tau|)^{2(\sigma+1/2)}.
\ena
Moreover, for $s=\sigma+i\tau$,
\bna
\widetilde{\varphi}(-s)
&=&N^{-s}V^{\dag}\left(-\frac{\zeta N}{Qqp^\lambda}, -s\right)\\
&\ll&N^{-\sigma}\min\left\{1, \left(\frac{N}{Qqp^\lambda|\tau|}\right)^j\right\}
\ena
for any $j\neq 0$, where $V^{\dag}(r,s)=\int_0^\infty V(y)e(-ry)y^{s-1}\mathrm{d}y$.
Here we used the bound $V^{\dag}(r,\sigma+i\tau)\ll\min\left\{1, \left(\frac{1+|r|}{|\tau|}\right)^j\right\}$.
Thus,
\bna
\Psi_\pm(x)
\ll \left(\frac{N}{Qqp^\lambda}\right)^{2+\varepsilon}\left(\frac{xQ^2q^2p^{2\lambda}}{N}\right)^{-\sigma}.
\ena
So the $m$-sum is supported on $m\leq\frac{q^2p^{2\kappa+\varepsilon}N}{Q^2q^2p^{2\lambda}}
\asymp p^{2\kappa-\lambda+\varepsilon}$ at the cost a negligible error. For small values of $m$,
we move the integration line to $\sigma=-1/2$ to get
\bna
\Psi_\pm\left(\frac{m}{q^2p^{2\kappa}}\right)
=\frac{ N^{1/2}m^{1/2}}{qp^\kappa}\mathfrak{I}^\pm\left(\frac{\pi^2m}{q^2p^{2\kappa}},q,\zeta\right)
\ena
where using \eqref{Q},
\bea\label{GL2 integral-1}
\mathfrak{I}^\pm(x,q,\zeta)=\frac{1}{2}\int_{\mathbb{R}}(Nx)^{-i\tau}
\rho_f^{\pm}\left(-\frac{1}{2}+i\tau\right)
V^{\dag}\left(-\frac{\zeta Q}{q},\frac{1}{2}-i\tau\right)\mathrm{d}\tau.
\eea
So we obtain
\bea\label{A-sum}
\mathscr{A}=\frac{ N^{1/2}}{\tau(\overline{\chi})}\sum_\pm
\sum_{c(\text{{\rm mod }} p^\kappa)}
\overline{\chi}(c)
\sum_{m\leq p^{2\kappa-\lambda+\varepsilon}}\frac{\lambda_f(m)}{m^{1/2}}
e\left(\pm\frac{m\overline{a^*}}{qp^\kappa}\right)
\mathfrak{I}^\pm\left(\frac{\pi^2m}{q^2p^{2\kappa}},q,\zeta\right)+O_A(\mathfrak{q}^{-A}).
\eea
Lemma \ref{A-voronoi} follows from \eqref{GL2 integral-1} and \eqref{A-sum}.

\section{Proof of Lemma \ref{B-voronoi}}
\setcounter{equation}{0}
\medskip
In this section we will apply the $\rm GL_2$ Voronoi formula to transform $\mathscr{B}$ in 
\eqref{B},
where
\bna
\mathscr{B}=
\sum_{n}\lambda_g(n) e\left(-\frac{an}{qp^\lambda}\right)\phi(n),
\ena
where $\phi(y)=W(y/N)e\left(-\zeta y/Qqp^\lambda\right)$.
Applying the $GL_2$ Voronoi formula in Lemma \ref{GL2 Voronoi formula}, we have
\bea
\mathscr{B}
=qp^{\lambda}\sum_\pm\sum_n\frac{\lambda_g(n)}{n}e\left(\mp\frac{\overline{a}n}{qp^\lambda}\right)
\Phi_\pm\left(\frac{n}{q^2p^{2\lambda}}\right),
\eea
where
\bea\label{GL2 integral-2}
\Phi_\pm(x)&=&\frac{1}{2\pi i}\int_{(\sigma)}(\pi^2 x)^{-s}
\rho_g^{\pm}(s)\widetilde{\phi}(-s) \mathrm{d}s,
\eea
By Stirling's formula, for $\sigma\geq -1/2$,
\bna
\rho_g^{\pm}(\sigma+i\tau)\ll(1+|\tau|)^{2(\sigma+1/2)}.
\ena
Moreover, for $s=\sigma+i\tau$,
\bna
\widetilde{\phi}(-s)&=&N^{-s}W^{\dag}\left(\frac{\zeta N}{Qqp^\lambda}, -s\right)\\
&\ll&N^{-\sigma}\min\left\{1, \left(\frac{N}{Qqp^\lambda|\tau|}\right)^j\right\}
\ena
for any $j\neq 0$, where we recall that $W^{\dag}(r,s)=\int_0^\infty W(y)e(-ry)y^{s-1}\mathrm{d}y$.
Here we used the bound $W^{\dag}(r,\sigma+i\tau)\ll\min\left\{1, \left(\frac{1+|r|}{|\tau|}\right)^j\right\}$.
Thus,
\bna
\Phi_\pm(x)
\ll \left(\frac{N}{Qqp^\lambda}\right)^{2+\varepsilon}\left(\frac{xQ^2q^2p^{2\lambda}}{N}\right)^{-\sigma}.
\ena
Thus the contribution from $n\geq N^{1+\epsilon}/Q^2\asymp p^{\lambda+\varepsilon}$ is negligible. 
For small values of $m$,
we move the integration line to $\sigma=-1/2$ to get
\bna
\Phi_\pm\left(\frac{n}{q^2p^{2\lambda}}\right)
=\frac{ N^{1/2}n^{1/2}}{qp^{\lambda}}\mathfrak{J}^\pm\left(\frac{\pi^2n}{q^2p^{2\lambda}},q,\zeta\right)
\ena
where using \eqref{Q},
\bea\label{GL2 integral-3}
\mathfrak{J}^\pm(x,q,\zeta)=\frac{1}{2}\int_{\mathbb{R}}(Nx)^{-i\tau}
\rho_g^{\pm}\left(-\frac{1}{2}+i\tau\right)
W^{\dag}\left(\frac{\zeta Q}{q},\frac{1}{2}-i\tau\right)\mathrm{d}\tau.
\eea
We conclude that
\bea\label{B-sum}
\mathscr{B}=N^{1/2}\sum_\pm
\sum_{n\leq p^{\lambda+\epsilon}}\frac{\lambda_g(n)}{n^{1/2}}
e\left(\mp\frac{n\overline{a}}{qp^\lambda}\right)
\mathfrak{J}^\pm\left(\frac{\pi^2n}{q^2p^{2\lambda}},q,\zeta\right)+O_A(\mathfrak{q}^{-A}).
\eea
Lemma \ref{B-voronoi} follows from \eqref{GL2 integral-3} and \eqref{B-sum}.

\medskip

\medskip

\section{Character sums}
\setcounter{equation}{0}
\medskip
In this section, we estimate the character sum in \eqref{character sum}, which by 
\eqref{character sum3} is
\bna
\mathcal{C}=e\left(\frac{-m n_1p^{\kappa-\lambda}\overline{p^{\lambda}}}{d}\right)
\mathcal{C}^*(\pm n_1,\pm n_2,\pm m),
\ena
where for $n_1,n_2,m\in \mathbb{Z}$,
\bna
\mathcal{C}^*(n_1,n_2,m)=p^{\kappa}\sum_{c(\text{{\rm mod }} p^\kappa)}\;
  \sideset{}{^*}\sum_{b_1 (\text{{\rm mod }} p^{\lambda})}\;
    \sideset{}{^*}\sum_{b_2 (\text{{\rm mod }} p^{\lambda})}
\overline{\chi}(\overline{c}-b_1p^{\kappa-\lambda})
\chi(\overline{c+mqd^{-1}}-b_2p^{\kappa-\lambda})
  e\left(\frac{-n_1\overline{q}\overline{b_1}}{p^\lambda}\right)
  e\left(\frac{n_2\overline{q}\overline{b_2}}{p^\lambda}\right).
\ena
So to estimate $\mathcal{C}$, we only need to evaluate $\mathcal{C}^*(n_1,n_2,m)$.

(1) Let $\lambda=2\alpha+\delta_1$, $\delta_1=0$ or 1 and $\alpha\geq 1$.
Write $b_1=a_1p^{\alpha+\delta_1}+a_2$, $a_1(\text{{\rm mod }}p^{\alpha})$,
$a_2(\text{{\rm mod }}p^{\alpha+\delta_1})$, and
$b_2=h_1p^{\alpha+\delta_1}+h_2$, $h_1(\text{{\rm mod }}p^{\alpha})$,
$h_2(\text{{\rm mod }}p^{\alpha+\delta_1})$.
Then
\bna
\mathcal{C}^*(n_1,n_2,m)
&=&p^{\kappa+2\alpha}\sum_{c(\text{{\rm mod }} p^\kappa)}\;
    \sideset{}{^*}\sum_{a_2 (\text{{\rm mod }} p^{\alpha+\delta_1})}\;
\overline{\chi}(\overline{c}-a_2p^{\kappa-\lambda})
  e\left(\frac{-n_1\overline{q}\overline{a_2}}{p^{\lambda}}\right)\\&&\times
\sideset{}{^*}\sum_{h_2 (\text{{\rm mod }} p^{\alpha+\delta_1})}
\chi(\overline{c+mqd^{-1}}-h_2p^{\kappa-\lambda})
  e\left(\frac{n_2\overline{q}\overline{h_2}}{p^{\lambda}}\right)
\\&&\times\frac{1}{p^{\alpha}}\sum_{a_1 (\text{{\rm mod }} p^{\alpha})}\;
\chi(1+\overline{\overline{c}-a_2p^{\kappa-\lambda}}a_1p^{\kappa-\alpha})
  e\left(\frac{n_1\overline{q}\overline{a_2^2}a_1}{p^{\alpha}}\right)
\\
&&\times \frac{1}{p^{\alpha}}
\sum_{h_1 (\text{{\rm mod }} p^{\alpha})}\;
\chi(1-\overline{\overline{c+mqd^{-1}}-h_2p^{\kappa-\lambda}}h_1p^{\kappa-\alpha})
  e\left(\frac{-n_2\overline{q}\overline{h_2^2}h_1}{p^{\alpha}}\right).
\ena
Recall $\chi$ is a primitive character of modulus $p^{\kappa}$ and $\kappa>\lambda\geq 2\alpha$.
Thus $\chi(1+zp^{\kappa-\alpha})$
is an additive character to modulus $p^{\alpha}$, so there exists an integer
$\eta$ (uniquely determined modulo $p^{\alpha}$), $(\eta,p)=1$, such that
$\chi(1+zp^{\kappa-\alpha})=\exp(2\pi i \eta z/p^{\alpha})$.
Thus
\bna
\mathcal{C}^*(n_1,n_2,m)
&=&p^{\kappa+2\alpha}\sum_{c(\text{{\rm mod }} p^\kappa)}\;
    \sideset{}{^*}\sum_{a (\text{{\rm mod }} p^{\alpha+\delta_1})\atop
\eta q \overline{n_1}a^2-p^{\kappa-\lambda}a+\overline{c}\equiv 0(\bmod p^{\alpha})}\;
\overline{\chi}(\overline{c}-ap^{\kappa-\lambda})
  e\left(\frac{-n_1\overline{q}\overline{a}}{p^{\lambda}}\right)\\&&\times
\sideset{}{^*}\sum_{h (\text{{\rm mod }} p^{\alpha+\delta_1})\atop
\eta q \overline{n_2}h^2-p^{\kappa-\lambda}h+\overline{c+mqd^{-1}}
\equiv 0(\bmod p^{\alpha})}
\chi(\overline{c+mqd^{-1}}-hp^{\kappa-\lambda})
  e\left(\frac{n_2\overline{q}\overline{h}}{p^{\lambda}}\right)
\ena
and $\mathcal{C}^*(n_1,n_2,m)$ vanishes unless $(n_1n_2,p)=1$.
Next, we consider the sum over $c$. Let $\kappa=\beta+\delta_2$,
$\delta_2=0$ or 1 and $\beta\geq \alpha\geq 1$.
Write $c=c_1p^{\beta+\delta_2}+c_2$, $c_1(\text{{\rm mod }}p^{\beta})$,
$c_2(\text{{\rm mod }}p^{\beta+\delta_2})$.
Then
\bna
\overline{\chi}(\overline{c}-ap^{\kappa-\lambda})
&=&\overline{\chi}(\overline{c_2}-ap^{\kappa-\lambda})
\chi(1+\overline{(\overline{c_2}-ap^{\kappa-\lambda})c_2^2}
c_1p^{\beta+\delta_2}),\\
\chi(\overline{c+mqd^{-1}}-hp^{\kappa-\lambda})&=&
\chi(\overline{c_2+mqd^{-1}}-hp^{\kappa-\lambda})
\chi(1-\overline{(\overline{c_2+ mqd^{-1}}-hp^{\kappa-\lambda})
(c_2\pm mqd^{-1})^2}
c_1p^{\beta+\delta_2}).
\ena
Recall $\chi$ is a primitive character of modulus $p^{\kappa}$ and $\kappa=2\beta+\delta_2$.
Thus $\chi(1+zp^{\beta+\delta_2})=\chi(1+zp^{\kappa-\beta})$
is an additive character to modulus $p^{\beta}$, so there exists an integer
$\eta'$ (uniquely determined modulo $p^{\beta}$), $(\eta',p)=1$, such that
$\chi(1+zp^{\beta+\delta_2})=\exp(2\pi i \eta' z/p^{\beta})$.
Therefore,
\bna
\mathcal{C}^*(n_1,n_2,m)
&=&p^{\kappa+2\alpha}\sideset{}{^*}\sum_{c_2(\text{{\rm mod }} p^{\beta+\delta_2})}\;
    \mathop{\sideset{}{^*}\sum_{a (\text{{\rm mod }} p^{\alpha+\delta_1})}\;
\sideset{}{^*}\sum_{h (\text{{\rm mod }} p^{\alpha+\delta_1})
}}_{
\eta q \overline{n_1}a^2-p^{\kappa-\lambda}a+\overline{c_2}\equiv 0(\bmod p^{\alpha})
\atop
\eta q \overline{n_2}h^2-p^{\kappa-\lambda}h+\overline{c_2+mqd^{-1}}
\equiv 0(\bmod p^{\alpha})}
f(c_2,a,h,m,d)\\&&\times
\sum_{c_1(\text{{\rm mod }} p^{\beta})}\;
e\left(\frac{\overline{(\overline{c_2}-ap^{\kappa-\lambda})c_2^2}
-\overline{(\overline{c_2+mqd^{-1}}-hp^{\kappa-\lambda})
(c_2+mqd^{-1})^2}}{p^{\beta}}\eta'c_1\right),
\ena
where
\bna
f(c_2,a,h,m,d)=\overline{\chi}(\overline{c_2}-ap^{\kappa-\lambda})
\chi(\overline{c_2+mqd^{-1}}-hp^{\kappa-\lambda})
  e\left(\frac{-n_1\overline{qa}+n_2\overline{qh}}
  {p^{\lambda}}\right).
\ena
By the orthogonality of additive characters, the last sum over $c_1$ vanishes unless
\bna
mqd^{-1}+ap^{\kappa-\lambda}c_2^2
+hp^{\kappa-\lambda}(c_2+ mqd^{-1})^2\equiv 0 (\text{{\rm mod }} p^{\beta}),
\ena
which implies $p^{\kappa-\lambda}|m$ if we assume $\lambda\geq \kappa/2$.
Write $m=p^{\kappa-\lambda}m'$ and further assume that $\lambda\leq 3\kappa/4$.
Then the above congruence is reduced to
$m'qd^{-1}+ac_2^2+hc_2^2\equiv 0(\text{{\rm mod }} p^{\beta-\kappa+\lambda})$.
Hence
\bna
\mathcal{C}^*(n_1,n_2,m)
\ll p^{\kappa+\beta+2\alpha} \mathop{\sideset{}{^*}\sum_{c_2(\text{{\rm mod }} p^{\beta+\delta_2})}\;
   \sideset{}{^*}\sum_{a (\text{{\rm mod }} p^{\alpha+\delta_1})}\;
\sideset{}{^*}\sum_{h (\text{{\rm mod }} p^{\alpha+\delta_1})
}1.}_{\substack{
\eta q \overline{n_1}a^2-p^{\kappa-\lambda}a+\overline{c_2}\equiv 0(\bmod p^{\alpha})\\
\eta q \overline{n_2}h^2-p^{\kappa-\lambda}h+\overline{c_2+p^{\kappa-\lambda}m'qd^{-1}}
\equiv 0(\bmod p^{\alpha})\\
m'qd^{-1}+ac_2^2+hc_2^2\equiv 0 (\text{{\rm mod }} p^{\beta-\kappa+\lambda})
}}
\ena
Moreover, we write $c_2=c_2'+p^{\alpha}c_2''$ with
$c_2'(\bmod \,p^{\alpha})$ and $c_2''(\bmod \,p^{\beta+\delta_2-\alpha})$. Then
\bea\label{middle-expression}
\mathcal{C}^*(n_1,n_2,m)
\ll p^{2\kappa+\alpha}
\mathop{\sideset{}{^*}\sum_{c_2'(\text{{\rm mod }} p^{\alpha})}\;
   \sideset{}{^*}\sum_{a (\text{{\rm mod }} p^{\alpha+\delta_1})}\;
\sideset{}{^*}\sum_{h (\text{{\rm mod }} p^{\alpha+\delta_1})
}1.}_{\substack{
\eta q \overline{n_1}a^2-p^{\kappa-\lambda}a+\overline{c_2'}\equiv 0(\bmod p^{\alpha})\\
\eta q \overline{n_2}h^2-p^{\kappa-\lambda}h+\overline{c_2'+p^{\kappa-\lambda}m'qd^{-1}}
\equiv 0(\bmod p^{\alpha})\\
m'qd^{-1}+ac_2'^2+hc_2'^2\equiv 0 (\text{{\rm mod }} p^{\beta-\kappa+\lambda})
}}
\eea

We further assume $\lambda\geq (2\kappa+1)/3$. Then the first two congruences
in \eqref{middle-expression} involving
$n_2$ imply that $\overline{n_2}h^2\equiv \overline{n_1}a^2(\bmod \,p^{\kappa-\lambda})$.
So for $p^{\beta-\kappa+\lambda}|m'$, the last congruence
in \eqref{middle-expression} imply $n_1\equiv n_2 (\bmod \,p^{\beta-\kappa+\lambda})$.
Trivially, for fixed $n_1$, $n_2$, $m'$ and $c_2'$, there are at most two roots
for $a$ and $h$ with modulus $p^{\alpha}$, respectively. Thus
\bna
\mathcal{C}^*(n_1,n_2,m)\ll p^{2\kappa+\lambda+\delta_1}.
\ena
This proves the first statement of the lemma.

(2) In the following we assume $m'=p^{\ell}m^*$ with $(m^*,p)=1$ and $\ell<\beta-\kappa+\lambda$.
To count the numbers of $c_2',a$ and $h$, we solve the three congruence
equations in \eqref{middle-expression}. Write $a=a_3+p^{\beta-\kappa+\lambda}a_4$,
$c_2'=c_3+p^{\beta-\kappa+\lambda}c_4$ and $h=h_3+p^{\beta-\kappa+\lambda}h_4$,
with $a_3,c_3,h_3 (\bmod \, p^{\beta-\kappa+\lambda})$,
$c_4(\bmod \, p^{\alpha-(\beta-\kappa+\lambda)})$ and
$a_4,h_4 (\bmod \, p^{\alpha+\delta_1-(\beta-\kappa+\lambda)})$.
Then the last congruence implies
\bea\label{congruence 1}
a_3\equiv -h_3-m'qd^{-1}\overline{c_3^2}\;(\bmod \, p^{\beta-\kappa+\lambda} )
\eea
which plugged into the first congruence equation in \eqref{middle-expression} yields
\bea\label{congruence 2}
\big(h_3+m'qd^{-1}\overline{c_3^2}\big)^2+\overline{\eta q }n_1\overline{c_3}
\equiv 0\;(\bmod \, p^{\beta-\kappa+\lambda}).
\eea
On the other hand, the second congruence equation in \eqref{middle-expression} implies
\bea\label{congruence 3}
h_3^2+\overline{\eta q }n_2\overline{c_3}
\equiv 0\;(\bmod \, p^{\beta-\kappa+\lambda}).
\eea
By \eqref{congruence 2} and \eqref{congruence 3}, one has
\bna
2m'qd^{-1}\overline{c_3^2}h_3+(m'qd^{-1})^2\overline{c_3^4}
+\overline{\eta q}(n_1-n_2)\overline{c_3}\equiv 0 \;(\bmod \, p^{\beta-\kappa+\lambda}),
\ena
which implies $p^{\ell}\| n_1-n_2$ and
\bea\label{congruence 4}
h_3\equiv -\overline{2\eta m^*q^2d^{-1}}p^{-\ell}(n_1-n_2)c_3
-\overline{2}p^{\ell}m^*qd^{-1}\overline{c_3^2}(\bmod \, p^{\beta-\kappa+\lambda-\ell}).
\eea
Plugging \eqref{congruence 4} into \eqref{congruence 3},  we obtain
\bna
\left(\overline{2\eta m^*q^2d^{-1}}p^{-\ell}(n_1-n_2)c_3^2
+\overline{2}p^{\ell}m^*qd^{-1}\right)^2+\overline{\eta q }n_2c_3^3
\equiv 0 (\bmod \, p^{\beta-\kappa+\lambda-\ell}).
\ena
It follows that for fixed $n_1,n_2$ and $m^*$, there are at most six roots modulo
$p^{\beta-\kappa+\lambda-\ell}$ for $c_3$. Fixing $c_3$, the relation
\eqref{congruence 4} fixes $h_3$ uniquely modulo
$p^{\beta-\kappa+\lambda-\ell}$. For fixed $c_3$ and $h_3$,
$a_3$ is uniquely determined modulo $p^{\beta-\kappa+\lambda}$  by \eqref{congruence 1}.
Now we fix $a_4$, then $c_4,h_4 (\bmod \, p^{\alpha-(\beta-\kappa+\lambda)})$
(and so $c_2'$, $a$, $h (\bmod \, p^{\alpha})$) are determined up to a constant.
We conclude that
\bna
\mathcal{C}^*(n_1,n_2,m)\ll p^{2\kappa+\alpha+2\ell +2\delta_1+\alpha-(\beta-\kappa+\lambda)}
\ll p^{5\kappa/2+2\ell +\delta_1+\delta_2/2}.
\ena

(3) For $m=0$, we have
\bna
\mathcal{C}^*(n_1,n_2,0)
=p^{\kappa}  \sideset{}{^*}\sum_{c(\text{{\rm mod }} p^\kappa)}\;
  \sideset{}{^*}\sum_{b_1 (\text{{\rm mod }} p^{\lambda})}\;
    \sideset{}{^*}\sum_{b_2 (\text{{\rm mod }} p^{\lambda})}
\overline{\chi}(c-b_1p^{\kappa-\lambda})
\chi(c-b_2p^{\kappa-\lambda})
  e\left(\frac{-n_1\overline{q}\overline{b_1}}{p^\lambda}\right)
  e\left(\frac{n_2\overline{q}\overline{b_2}}{p^\lambda}\right).
\ena
By the Fourier expansion of $\chi$ in the terms of additive characters (see (3.12) in \cite{IK}),
we have
\bna
\mathcal{C}^*(n_1,n_2,0)
&=&\frac{p^{\kappa}}{\tau(\chi)\tau(\overline{\chi})}
\sum_{\ell_1(\bmod p^\kappa)}\chi(\ell_1)
\sum_{\ell_2(\bmod p^\kappa)}\overline{\chi}(\ell_2)
 \sideset{}{^*}\sum_{b_1 (\text{{\rm mod }} p^{\lambda})}
e\left(\frac{-b_1\ell_1-n_1\overline{q}\overline{b_1}}{p^{\lambda}}\right)
 \\&&\times  \sideset{}{^*}\sum_{b_2 (\text{{\rm mod }} p^{\lambda})}
 e\left(\frac{-b_2\ell_2+n_2\overline{q}\overline{b_2}}{p^{\lambda}}\right)
\sideset{}{^*}\sum_{c(\text{{\rm mod }} p^\kappa)}\;
e\left(\frac{c(\ell_1+\ell_2)}{p^\kappa}\right).
\ena
The sum over $c$ equals
\bna
S(0,\ell_1+\ell_2; p^\kappa)=
\sum_{d|(\ell_1+\ell_2,p^{\kappa})}d\mu\left(\frac{p^\kappa}{d}\right)
=p^{\kappa}\mathbf{1}_{\ell_2\equiv -\ell_1(\text{{\rm mod }} p^{\kappa})}
-p^{\kappa-1}\mathbf{1}_{\ell_2\equiv -\ell_1(\text{{\rm mod }} p^{\kappa-1})}.
\ena
Note that
\bea\label{vanish1}
\sum_{\ell_2(\bmod p^\kappa)\atop \ell_2\equiv -\ell_1(\text{{\rm mod }} p^{\kappa-1})}
\overline{\chi}(\ell_2)e\left(\frac{-b_2\ell_2}{p^{\lambda}}\right)
=e\left(\frac{-b_2\ell_1}{p^{\lambda}}\right)\overline{\chi}(-\ell_1)
\sum_{a(\bmod p)}\chi(1+\overline{\ell_1}p^{\kappa-1}a).
\eea
Recall $\chi$ is a primitive character of modulus $p^{\kappa}$.
Thus $\chi(1+zp^{\kappa-1})$ is an additive character to the modulus $p$,
so there exists an integer
$\eta''$ (uniquely determined modulo $p$), $(\eta'',p)=1$, such that
$\chi(1+zp^{\kappa-1})=\exp(2\pi i \eta'' z/p)$.
Therefore the $a$-sum in \eqref{vanish1} vanishes and
\bna
\mathcal{C}^*(n_1,n_2,0)
&=&\frac{p^{2\kappa}\overline{\chi}(-1)}{\tau(\chi)\tau(\overline{\chi})}\;
 \sideset{}{^*}\sum_{b_1 (\text{{\rm mod }} p^{\lambda})}
e\left(\frac{-n_1\overline{q}\overline{b_1}}{p^{\lambda}}\right)
\sideset{}{^*}\sum_{b_2 (\text{{\rm mod }} p^{\lambda})}
 e\left(\frac{n_2\overline{q}\overline{b_2}}{p^{\lambda}}\right)
 \sum_{\ell_1(\bmod p^\kappa)}
 e\left(\frac{(b_2-b_1)\ell_1}{p^{\lambda}}\right)\\
&=&p^{2\kappa}\;
 \sideset{}{^*}\sum_{b_1 (\text{{\rm mod }} p^{\lambda})}
e\left(\frac{-n_1+ n_2}{p^{\lambda}}b_1\right)\\
&=&p^{2\kappa}\sum_{d|(n_1-n_2,p^{\lambda})}d\mu\left(p^{\lambda}/d\right).
\ena
Thus the last statement of the lemma follows.

\bigskip

\noindent
{\sc Acknowledgements.}
This work is partly supported by NSFC (Nos. 11871306, 12031008).

\bigskip

{\sc \small School of Mathematics and Statistics, Shandong University, Weihai,
Weihai, Shandong 264209, China}

{\footnotesize {\it E-mail address:} qfsun@sdu.edu.cn}


\begin{thebibliography}{100}

\bibitem{ASS}{}

R. Acharya, P. Sharma and S. K. Singh,
{\it $t$-aspect subconexity for $\rm GL(2)\times GL(2)$ $L$-function},
arXiv:2011.01172.


\bibitem{BM}{}
 V. Blomer and D. Mili\'{c}evi\'{c}, {\it $p$-adic analytic twists and strong subconvexity},
 Ann. Sci. \'{E}c. Norm. Sup\'{e}r. (4) 48 (2015), no. 3, 561-605.



\bibitem{IK}{}
H. Iwaniec and E, Kowalski, {\it Analytic number theory},
Amercian Mathematical Society Colloquium Publications 53, Amercian Mathematical Society, Providence, RI, 2004.

\bibitem{KPY}{}
E. M. Kiral, I. Petrow and Matthew P. Young,
{\it Oscillatory integrals with uniformity in parameters},
J. Th\'{e}or. Nombres Bordeaux 31 (2019), no. 1, 145-159.

\bibitem{Le}{}
D. Letang, {\it
Subconvexity bounds for automorphic L-functions on $\rm GL2$},
Thesis (Ph.D.), University of Minnesota, 2009.

\bibitem{LS}{}
Y. Lin and Q. Sun, {\it
Analytic twists of $\text{GL}_3 \times \text{GL}_2$ automorphic forms},
DOI:10.1093/imrn/rnaa348, Int. Math. Res. Not., to appear.



\bibitem{Mili}{}
D. Mili\'{c}evi\'{c}, {\it Sub-Weyl subconvexity for Dirichlet $L$-functions to prime power moduli},
Compos. Math. 152 (2016), no. 4, 825-875.



\bibitem{MS}{}
S.D. Miller and W. Schmid,
{\it Automorphic distributions, $L$-functions, and Voronoi summation for $GL(3)$},
Ann. of Math. (2) 164 (2006), no. 2, 423-488.


\bibitem{Mol}{}
G. Molteni,
{\it Upper and lower bounds at $s=1$ for certain Dirichlet series with Euler product},
Duke Mathematical Journal 111 (2002), no. 1, 133-158.



\bibitem{Munshi 1}{}
R. Munshi,
{\it The circle method and bounds for $L$-functions, II:
Subconvexity for twists of $GL(3)$ $L$-functions},
Amer. J. Math. 137 (2015), no. 3, 791-812.


\bibitem{Munshi 2}{}
R. Munshi,
{\it The circle method and bounds for $L$-functions-III:
$t$-aspect subconvexity for $GL(3)$ $L$-functions},
J. Amer. Math. Soc. 28 (2015), no. 4, 913-938.



\bibitem{Munshi-Singh}{}
R. Munshi and S. K. Singh
{\it Weyl bound for p-power twist of $\rm GL(2)$ $L$-functions},
Algebra Number Theory 13 (2019), no. 6, 1395-1413.


\bibitem{SZ}{}
Q. Sun and R. Zhao,
{\it Bounds for $\rm GL_3$ $L$-functions in depth aspect,}
Forum Math. 31 (2019), no. 2, 303-318.


\end{thebibliography}
\end{document}